\documentclass{amsart}
\usepackage{enumerate}
\usepackage{color}

\newtheorem{theorem}{Theorem}[section]
\newtheorem{lemma}[theorem]{Lemma}
\theoremstyle{definition}
\newtheorem{example}{Example}[section]

\newtheorem{corollary}[theorem]{Corollary}
\newtheorem{proposition}[theorem]{Proposition}

\newtheorem*{rep@theorem}{\rep@title}
\newcommand{\newreptheorem}[2]{%
\newenvironment{rep#1}[1]{%
 \def\rep@title{#2 \ref{##1}*}%
 \begin{rep@theorem}}%
 {\end{rep@theorem}}}

\newreptheorem{theorem}{Theorem}
\newtheorem{problem}{Problem}

\theoremstyle{remark}
\newtheorem{remark}{Remark}[section]

\newcommand{\F}{ \ensuremath{ \mathbb{F}}}
\newcommand{\eS}{ \ensuremath{ \mathbb{S}}}
\newcommand{\Z}{ \ensuremath{ \mathbb Z}}
\newcommand{\Tr}{ \ensuremath{ \mathrm{Tr}}}
\newcommand{\Aut}{ \ensuremath{ \mathrm{Aut}}}
\newcommand{\bbM}[3]{\mathbb{M}_{#1\times#2}(#3)}
\newcommand{\rZ}[1]{\mathbb{Z}/#1\mathbb{Z}}
\newcommand{\cy}[1]{\Z_{#1}}

\numberwithin{equation}{section}

\begin{document}
\title{$(2^n,2^n,2^n,1)$-relative difference sets and their representations}
\author{Yue Zhou}
\begin{abstract}
    We show that every $(2^n,2^n,2^n,1)$-relative difference set $D$ in $\Z_4^n$ relative to $\Z_2^n$  can be represented by a polynomial $f(x)\in \F_{2^n}[x]$, where $f(x+a)+f(x)+xa$ is a permutation for each nonzero $a$. We call such an $f$ a planar function on $\F_{2^n}$. The projective plane $\Pi$ obtained from $D$ in the way of Ganley and Spence \cite{ganley_relative_1975} is coordinatized, and we obtain necessary and sufficient conditions of $\Pi$ to be a presemifield plane. We also prove that a function $f$ on $\F_{2^n}$ with exactly two elements in its image set and $f(0)=0$ is planar, if and only if, $f(x+y)=f(x)+f(y)$ for any $x,y\in\F_{2^n}$.
\end{abstract}
\maketitle

\section{Introduction}\label{se_introduction}

Let $G$ be a group of order $v=mn$, let $N$ be a subgroup of $G$ satisfying $\# N=n$.
A $k$-subset $D$ of $G$  is called a \emph{relative difference set} with parameters $(m,n,k,\lambda)$ (abbreviated to $(m,n,k,\lambda)$-RDS) relative to $N$ if the list of differences of $D$ covers every element of $G\setminus N$ exactly $\lambda$ times, and no element of $N\setminus \{0\}$; $N$ is called the \emph{forbidden subgroup}. We call $D$ \emph{abelian} or \emph{cyclic} if $G$ has the respective property. A $(v,1,k,\lambda)$-RDS is just a $(v,k,\lambda)$-difference set.

A special case of relative difference sets appeared first in the work of Bose \cite{bose_affine_1942}, and the term ``relative difference set'' was introduced by Butson \cite{butson_relations_1963}. The main motivation to study relative difference sets comes from the fact that the existence of an $(m,n,k,\lambda)$-RDS in $G$ relative to $N$ is equivalent to the existence of a \emph{divisible design} with the same parameters, admitting $G$ as a regular automorphism group. If $N$ is normal in $G$, then $N$ acts regularly on each point class of this design. For further results and references see \cite{pott_survey_1996}. For a comprehensive introduction to difference sets see Chapter VI of \cite{beth_design_1999}.

A \emph{projective plane} consists of a set of lines $\mathcal{L}$ and a set of points $\mathcal{B}$ such that for any two given distinct points (resp. lines), there is exactly one line (resp. point) incident with both of them, and there are four points such that no line is incident with more than two of them. By removing one line and the points on this line from a projective plane, we get an \emph{affine plane} which forms a divisible design.

Two projective planes are called \emph{isomorphic} if there is a permutation on $\mathcal{L}$ and a permutation on $\mathcal{B}$ such that the incidence relationship is preserved. A \emph{collineation} (or \emph{automorphism}) of a projective plane is an isomorphism of the plane onto itself.

Every projective plane $\Pi$ can be coordinatized with a  \emph{planar ternary ring}, abbreviated to \emph{PTR}, in an appropriate way. Different ways of coordinatization may lead to different PTRs. Moreover, if $\Pi$ has some special types of collineation group, it may give rise to some extra algebraic properties of the PTR. For instance, if $G$ fixes a line $L_\infty$ pointwise and acts regularly on the other points, then using the coordinatization method in \cite{hughes_projective_1973} we can get a special PTR which is actually a \emph{quasifield}; this line $L_\infty$ is called a \emph{translation line} and $\Pi$ is a \emph{translation plane}. For the definitions of PTRs, quasifields and other basic fact of projective planes, we refer to \cite{hughes_projective_1973}.

As shown by Ganley and Spence \cite[Theorem 3.1]{ganley_relative_1975}, if $D\subseteq G$ is an RDS  relative to a normal subgroup $N$ with parameters $(m,n,k,1)=(q,q,q,1),(q+1,q-1,q,1)$ or $(q^2+q+1, q^2-q,q^2,1)$, then $D$ can be uniquely extended to a projective plane $\Pi$. Furthermore, $G$ acts quasi-regularly on the points and lines of $\Pi$, and the point orbits (and line orbits) of $\Pi$ under $G$ are the cases $(b)$, $(d)$ and $(e)$  respectively in the Dembowski-Piper classification \cite{dembowski_quasiregular_1967}. For case $(b)$, in which $D$ is a $(q,q,q,1)$-RDS in $(G,+)$ relative to $N$, the projective plane $\Pi$ can be obtained in the following way \cite{ganley_relative_1975}:
\begin{description}
  \item[affine points] $g\in G$;
  \item[lines] $D+g := \{d+g: d \in D\}$, and the distinct cosets of $N$: $N+g_1,\dots,N+g_q$ and a new line $L_\infty$ defined as a set of extra points;
  \item[points on $L_\infty$] $(g_i)$ defined by the parallel class $\{D+g_i+h: h\in N\}$ and $(\infty)$ defined by the parallel class $\{N+g_1,\dots,N+g_q\}$.
\end{description}
Under the action of $G$, the points of $\Pi$ form three orbits: the affine points, $\{(\infty)\}$, and the points on $L_\infty$ except for $(\infty)$. The lines of $\Pi$ also have three orbits: $\{D+g:g\in G\}$, $\{L_\infty\}$, and all the lines incident with $(\infty)$ except for $L_\infty$. When $G$ is commutative, $\Pi$ is called a \emph{shift plane} and $G$ is its \emph{shift group}.

Let $\Z_m$ denote the cyclic group of order $m$. Ganley \cite{ganley_paper_1976} and Blokhuis, Jungnickel, Schmidt \cite{blokhuis_proof_2002} proved the following result:
\begin{theorem}\label{th_(q,q,q,1)-RDS}
    Let $D$ be a $(q,q,q,1)$-RDS in an abelian group $G$ of order $q^2$.
    \begin{itemize}
      \item\cite{ganley_paper_1976}  If $q$ is even, then $q$ is a power of $2$ (say $q=2^n$), $G\cong\Z_4^n$ and the forbidden subgroup $N\cong\Z_2^n$ (see also \cite{jungnickel_theorem_1987} for a short proof);
      \item\cite{blokhuis_proof_2002} If $q$ is odd, then $q$ is a prime power (say $q=p^n$) and the rank of $G$, i.e. the smallest cardinality of a generating set for $G$,  is at least $n+1$.
    \end{itemize}
\end{theorem}

For $G$ abelian and $q=p^n$ odd, all the known $(q,q,q,1)$-RDS are subsets of $(\F_p^{2n},+)$, a $2n$-dimensional vector space over a finite field $\F_{p}$. Moreover, if $D$ is a $(q,q,q,1)$-RDS in $G$, then all the known examples can be expressed as:
$$D = \{(x,f(x)): x\in \F_{p^n})\},$$
where $f$ is a so-called \emph{planar function} from $\F_{p^n}$ to itself. A function $f:\F_{p^n} \rightarrow\F_{p^n}$ is \emph{planar} if the mapping
$$x\mapsto f(x+a)-f(x),$$
is a permutation for all $a\neq 0$. Until now, except for one special family, every known planar function on $\F_{p^n}$ can be written, up to adding affine terms, as a \emph{Dembowski-Ostrom Polynomial} $f(x)=\sum_{i,j=0}^{n-1}c_{ij}x^{p^i+p^j}$. It produces a commutative distributive quasifield plane \cite{dembowski_planes_1968}, namely a commutative semifield plane, on which the multiplication is defined by:
\begin{equation}\label{eq:x star y}
    x\star y := \frac{1}{2}(f(x+y)-f(x)-f(y))=\sum_{i,j=0}^{n-1}\frac{c_{ij}+c_{ji}}{2}x^{p^i}y^{p^j}.
\end{equation}
The only known exception was given by Coulter and Matthews \cite{Coulter-Matthews1997}:
$$f(x)=x^{\frac{3^\alpha+1}{2}}$$
over $\F_{3^n}$ provided that $\alpha$ is odd and $\gcd(\alpha, n)=1$. The plane produced by the corresponding RDS is not a translation plane \cite{Coulter-Matthews1997}.

A \emph{semifield} is a field-like algebraic structure on which the associative property of multiplication does not necessarily hold. When the existence of a multiplicative identity is not guaranteed, then it is called \emph{presemifield}. Let $\mathbb{S}=(\mathbb{F}_{p^n},+,\star )$ be a semifield. The subsets
\begin{equation*}
\begin{aligned}
  N_l(\mathbb{S})=\{a\in \mathbb{S}: (a\star x)\star y=a\star (x\star y) \text{ for all }x,y\in \mathbb{S}\},\\
  N_m(\mathbb{S})=\{a\in \mathbb{S}: (x\star a)\star y=x\star (a\star y) \text{ for all }x,y\in \mathbb{S}\},\\
  N_r(\mathbb{S})=\{a\in \mathbb{S}: (x\star y)\star a=x\star (y\star a) \text{ for all }x,y\in \mathbb{S}\},
\end{aligned}
\end{equation*}
are called the \emph{left, middle} and \emph{right nucleus} of $\mathbb{S}$, respectively. A recent survey about finite semifields can be found in \cite{lavrauw_semifields_2011}. One important fact that we need here is that the order of every finite (pre)semifield is a power of a prime and its additive group is  elementary abelian.  Now let $\eS=(\F_{p^n}, +, *)$ be a finite presemifield. Since $x* y$ is additive with respect to both variables $x$ and $y$, it can be written as a \emph{$p$-polynomial} map:
$$x* y =\sum_{i,j=0}^{n-1} a_{ij} x^{p^i} y^{p^j},$$
where $a_{ij}\in\F_{p^n}$. For $p$ odd, a commutative semifield can be obtained through (\ref{eq:x star y}) from a $(p^n,p^n,p^n,1)$-RDS. If $p=2$ and a $(2^n,2^n,2^n,1)$-RDS defines a commutative semifield, then we will show later how to get its multiplication in the form of a $p$-polynomial. If a plane $\Pi$ can be coordinatized with a semifield, $\Pi$ is called a \emph{semifield plane}. It can be shown that $\Pi$ is a semifield plane if and only if $\Pi$ and its dual plane are both translation planes. In some papers this condition is used as the definition of semifield planes, see for example \cite{lavrauw_semifields_2011}.

The rest of this paper is organized as follows:  In Section \ref{se_representations}, we give two representations of a $(2^n,2^n,2^n,1)$-RDS in $\Z_4^n$. Then in Section \ref{se_coordination}, we coordinatize the plane $\Pi$ from a $(2^n,2^n,2^n,1)$-RDS, and present necessary and sufficient conditions that $\Pi$ is a semifield plane. In Section \ref{se_nonexistence}, we consider planar functions $f$ with $\#\mathrm{Im}(f)=2$. In Section \ref{se_component}, we look at the component functions of the $\F_2^n$-representations.

\section{Representations of $D$}\label{se_representations}
Theorem \ref{th_(q,q,q,1)-RDS} shows that every $(2^n,2^n,2^n,1)$-RDS is a subset of $\Z_4^n$ relative to $2\Z_4^n\cong\Z_2^n$. For convenience, we will always say ``$\Z_4^n$ relative to $\Z_2^n$'' instead of ``$\Z_4^n$ relative to $2\Z_4^n$'' in the rest of this paper. Before considering these relative difference sets, we introduce some notation for the elements of $\Z_4^n$. Define an embedding $\psi :\F_2\rightarrow \Z_4=\{0,1,2,3\}$ by $0^\psi=0$ and $1^\psi=1$, and define $\Psi :\F_2^n\rightarrow \Z_4^n$ by
$$\Psi(x_0,x_1,\cdots,x_{n-1})=(\psi(x_0),\psi(x_1),\cdots,\psi(x_{n-1})).$$
As $0^\psi+0^\psi=0$, $0^\psi+1^\psi=1$ and $1^\psi+1^\psi=0+2\cdot 1^\psi$, we have
\begin{equation}\label{eq:basic_psi}
    x^\psi+y^\psi= (x+y)^\psi + 2(xy)^\psi
\end{equation}
for $x,y\in\cy{2}$.
Every $\xi\in\Z_4^n$ can be uniquely expressed as a 2-tuple
$$\xi = \lfloor a,b\rfloor:=a^\Psi+2b^\Psi,$$
where $a,b\in\F_2^n$. For instance $3\in\Z_4$ is $\lfloor 1,1 \rfloor$, and the forbidden subgroup can be written as $\{\lfloor 0,b \rfloor: b \in\F_2^n \}$. Furthermore, for $\lfloor a,b\rfloor ,\lfloor c,d\rfloor \in \Z_4^n$, we have
$$\lfloor a,b\rfloor +\lfloor c,d\rfloor =\lfloor a+ c, b+ d + (a\odot c)\rfloor ,$$
where $a\odot c:= (a_0c_0,\dots,a_{n-1} c_{n-1})$
for $a=(a_0,\dots,a_{n-1})$ and $c=(c_0,\dots,c_{n-1})$. It is clear from the context whether the symbol ``$+$'' refers to the addition of $\Z_{4}^n$ or the addition of $\Z_{2}^n$.

\begin{remark}
    In the language of group theory, $\Z_4^n$ can be viewed as an extension of $\Z_2^n$ by $\Z_2^n$. The set $\{\lfloor a, 0 \rfloor: a\in\F_2^n\}$ forms a transversal for the subgroup $\Z_2^n$ in $\Z_4^n$ and $\odot : \Z_2^n\times\Z_2^n \rightarrow \Z_2^n$ is the corresponding factor set (or cocycle). Factor sets form an important tool for many combinatorial objects, see \cite{horadam_hadamard_2007} for the applications to Hadamard matrices and RDSs.
\end{remark}

Let $D$ be a transversal of the forbidden subgroup $\Z_2^n$ in $\Z_4^n$. Then we can write every element of $D$ as
\begin{equation}\label{eq:d,h(d)}
  \lfloor d,h(d)\rfloor = d^\Psi+2h(d)^\Psi,
\end{equation}
where $h$ is a mapping from $\Z_4^n/ \Z_2^n$ to $\Z_2^n$. When $D$ is a $(2^n,2^n,2^n,1)$-RDS, $D$ is also a transversal, otherwise the difference list of $D$ will contain some element of the forbidden subgroup $\Z_2^n$. Let $\lfloor a,b\rfloor \in \Z_4^n$ and $a\neq 0$. As there is exactly one element of $(D+\lfloor a,b\rfloor)\cap D$, the equation
$$\lfloor d+ a, h(d)+ b+ (d\odot a)\rfloor =\lfloor d', h(d')\rfloor ,$$
holds for exactly one pair $(d,d')$, which means that the mapping
\begin{equation}\label{eq:Delta_h}
  \Delta_{h,a} : d \mapsto h(d+ a) + h(d)+ (d\odot a)
\end{equation}
is bijective for each $a\neq 0$. Conversely, if $h$ is  such that $\Delta_{h,a}$ is a permutation for all nonzero $a$, then $D$ is a $(2^n,2^n,2^n,1)$-RDS.

\begin{remark}
    In fact, $\Delta_{h,a}$ already appears in \cite{yutaka_factor_1991} by Hiramine in the form of factor sets. However, our main idea here is to use it to derive special types of functions over finite fields. It is more convenient and easier to consider the properties and constructions of $(2^n,2^n,2^n,1)$-RDS, since every mapping from a finite field to itself can be expressed as a polynomial.
\end{remark}

As $h:\Z_4^n/\Z_2^n\rightarrow\Z_2^n$ defined by $D$ by (\ref{eq:d,h(d)}) can be viewed as a mapping from $\F_2^n$ to itself, we call $h$ the \emph{$\F_2^n$-representation} of the transversal $D$. We will use it to introduce the polynomial representation over a finite field.

Let $B=\{\xi_i: i=0,1,\cdots, n-1\}$ be a basis of $\F_{2^n}$ over $\F_2$. We can view both $h$ and $\Delta_{h,a}$ as mappings from $\F_{2^n}$ to itself using this basis $B$, and express them as polynomials $h_B, \Delta_{h_B,a}\in \F_{2^n}[x]$. Furthermore, for $x=\sum_{i=0}^{n-1}x_i\xi_i$ and $y=\sum_{i=0}^{n-1}y_i\xi_i\in\F_{2^n}$, we define
$$x\odot_B y:=\sum_{i=0}^{n-1} x_i y_i \xi_i,~~\mu_B(x):=\sum_{i< j} x_i x_j \xi_i\xi_j,$$
and
$$f_B(x):= h_B(x)^2+\mu_B(x).$$
Then
\begin{align*}
    \nabla_{f_B,a}(x):& =f_B(x+a)+f_B(x)+f_B(a)+xa\\
                      & = (h_B(x+a) + h_B(x) + h_B(a))^2+(\mu_B(x+a) + \mu_B(x) + \mu_B(a))+xa\\
                      & = (h_B(x+a) + h_B(x) + h_B(a))^2+(x\odot_B a)^2\\
                      & = (\Delta_{h_B,a}(x) + h_B(a))^2,
\end{align*}
and it is straightforward to see that for each $a\neq 0$, (\ref{eq:Delta_h}) is a bijection if and only if $\nabla_{f_B,a}(x)$ acts on $\F_{2^n}$ as a permutation. We call $f_B(x)\in\F_{2^n}[x]$ the \emph{$\F_{2^n}$-representation} of $D$ with respect to the basis $B$. We summarize the above results in the following theorem.
\begin{theorem}\label{th_represetation}
    Let $D\subseteq \Z_4^n$ be a transversal of $\Z_2^n$ in $\Z_4^n$, let $B$ be a basis of $\F_{2^n}$ over $\F_2$, let $h$ be the $\F_2^n$-representation of $D$ and let $f$ be the $\F_{2^n}$-representation $f_B(x)$ with respect to $B$. Then the following statements are equivalent:
    \begin{enumerate}
        \item $D$ is an RDS in $\Z_4^n$ relative to $\Z_2^n$;
        \item $\Delta_{h,a}$ is bijective for each $a\neq 0$;
        \item $\nabla_{f_B,a}(x)$ is a permutation polynomial for each $a\neq 0$.
    \end{enumerate}
\end{theorem}

\begin{remark}\label{remark_no_linear_term}
Let $u,v\in\F_{2^n}$, let $(c_0,c_1,\dots,c_{n-1})$ be a nonzero vector in $\F_2^n$ and let $d_0\in\F_2$. If $g_B(x)= f_B(x) + u x^{2^i}+v$ for some $i$, then
\begin{align*}
  \nabla_{g_B,a}(x)&=g_B(x+a)+g_B(x)+g_B(a)+xa\\
  &=f_B(x+a)+f_B(x)+f_B(a)+v+xa\\
  &=\nabla_{f_B,a}(x)+v
\end{align*}
which means that adding affine terms $u x^{2^i}+v$ to $f_B(x)$ does not change the permutation properties of $\nabla_{f_B,a}$. By a similar argument, we see that adding $\sum_i c_i x_i+d_0$ to any coordinate functions $h_i$ of $h$ does not change the permutation properties of $\Delta_{h,a}$ either. If there is no $u x^{2^i}$ or constant term in $f_B(x)$ (resp. no linear or constant term in every coordinate function of $h$), we call $f_B$ a \emph{normalized} $\F_{2^n}$-representation (resp. $h$ a \emph{normalized} $\F_2^n$-representation) of $D$. They are similar to Dembowski-Ostrom polynomials in $\F_{p^n}[x]$ with odd $p$, because both of them have no linear or constant term.
\end{remark}

Since we can always use an RDS in $\Z_4^n$ relative to $\Z_2^n$ to construct a plane of order $2^n$, we call $f:\F_{2^n} \rightarrow \F_{2^n}$ a \emph{planar function} if for each $a\neq 0$,
$$f(x+a)+f(x)+xa$$
is a permutation on $\F_{2^n}$. As every mapping from $\F_{2^n}$ to itself can be written as a polynomial in $\F_{2^n}[x]$, the corresponding polynomial of a planar function is called \emph{planar polynomial}.
\begin{remark}
    For $p$ odd, Dembowski and Ostrom defined planar functions in the following way: $f:\F_{p^n} \rightarrow \F_{p^n}$ is \emph{planar} if, for all $a\neq 0$,
    $$f(x+a)-f(x)$$
    is a permutation on $\F_{p^n}$, see \cite{dembowski_planes_1968}. They showed that $\{(x,f(x):x \in \F_{p^n})\}$ is a $(p^n,p^n,p^n,1)$-RDS in $\Z_p^{2n}$ relative to $\Z_p^{n}$ if and only if $f$ is planar.  Planar functions for $p$ odd are also called \emph{perfect nonlinear functions}; they are useful in the construction of S-boxes in block ciphers to resist some attacks, see \cite{nyberg_perfect_1991}. However, they do not exist on $\F_{2^n}$, since $f(x+a)-f(x)=f((x+a)+a)-f(x+a)$. Therefore it does not make sense to extend the classical definition of planar functions to the case $p=2$ directly.
\end{remark}

The advantage of the above representations then is that we can apply finite fields theory to construct and analyze $(2^n,2^n,2^n,1)$-RDS in $\Z_4^n$. Here are some examples:
\begin{example}\label{ex_f=0}
    For each positive integer $n$, every affine mapping, especially $f(x)=0$, is a planar function on $\F_{2^n}$. The corresponding plane is a Desarguesian plane.
\end{example}

\begin{example}\label{ex_KWsemifield}
    Assume that we have a chain of fields $\F=\F_0\supset \F_1\supset\cdots\supset \F_n$ of characteristic $2$ with $[\F:\F_n]$ odd and corresponding trace mappings $\Tr_i:\F\rightarrow \F_i$. In \cite{kantor2003}, Kantor presented commutative presemifields $\mathbf{B}((\F_i)_0^n,(\zeta_i)_1^n)$ on which the multiplication is defined as:
    \begin{equation}\label{eq:Kantor}
        x*y=xy+(x\sum_{i=1}^n \Tr_i(\zeta_iy)+y\sum_{i=1}^n \Tr_i(\zeta_ix))^2,
    \end{equation}
    where $\zeta_i\in \F^*$, $1\le i\le n$. These semifields are related to a subfamily of the symplectic spreads constructed in \cite{kantor_symplectic_2004}. In (\ref{eq:Kantor}) we have $(x\sum_{i=1}^n \Tr_i(\zeta_ix))^2$ as a planar function on $\F$. Note that this semifield is a generalization of Knuth's binary semifields \cite{knuth_class_1965}, on which the multiplication is defined as:
    \begin{equation}\label{eq:knuth}
      x*y=xy+(x\Tr(y)+y\Tr(x))^2,
    \end{equation}
    corresponding to the presemifields $\mathbf{B}((\F_i)_0^1,(1))$. The planar function derived from Knuth's semifield is $(x\Tr(x))^2$.
\end{example}

Next, we consider the equivalence between $(2^n,2^n,2^n,1)$-RDSs and how an equivalence transformation affects the $\F_2^n$-representation $h$ of a $(2^n,2^n,2^n,1)$-RDS.

Let $D_1$ and $D_2\subseteq G$ be two $(2^n,2^n,2^n,1)$-RDSs. They are \emph{equivalent} if there exists some $\alpha\in\Aut(G)$ and $a\in G$ such that $\alpha(D_1)=D_2+a$. When $G$ is abelian, every element of $\Aut(G)$ can be expressed by a matrix, see \cite{Hillar06,Ranum1907} for details. Let $\rZ{m}$ denote the residue class ring of integers modulo $m$, and let $\bbM{m}{n}{R}$ denote the set of $m\times n$ matrices with entries in a ring $R$. In the case that $G$ is $\Z_4^n$, the corresponding result is:
\begin{lemma}\label{le:aut(z4)}
    Let mapping $\rho:\bbM{n}{n}{\rZ{4}}\rightarrow\mathrm{Hom}(\cy{4}^n,\cy{4}^n)$ be defined as:
    $$\rho(L) (a_0,\dots,a_{n-1})= L(a_0,\dots,a_{n-1})^T.$$
    Then $\rho$ is surjective. Furthermore, $\rho(L)$ is an automorphism of $\cy{4}^n$ if and only if $(L \bmod 2)$ is invertible.
\end{lemma}
For instance, let $\beta$ be an element of $\mathrm{Hom}(\cy{4}^2,\cy{4}^2)$ defined by $\beta(1,0)=(1,2)$ and $\beta(0,1)=(1,1)$. Then we take matrix
$$L=\left(
         \begin{array}{cc}
           1 & 1 \\
           2 & 1 \\
         \end{array}
       \right),
$$
and it follows that $\rho(L)=\beta$. As $(L \bmod 2)$ is invertible, $\beta$ is an automorphism.

Define $*_h: \F_2^n\times\F_2^n\rightarrow \F_2^n$ by
\begin{equation}\label{eq:*}
      x*_h y:=h(x+y)+h(x)+h(y)+x\odot y,~~~x,y\in \F_2^n,
\end{equation}
where $x\odot y=(x_0y_0,x_1y_1,\cdots, x_{n-1}y_{n-1})$.
\begin{theorem}\label{th_equiRDS}
    Let $D_1$ and $D_2$ be $(2^n,2^n,2^n,1)$-RDS in $\Z_4^n$ relative to $\Z_2^n$, and let $h_1, h_2:\F_2^n\rightarrow \F_2^n$ be their normalized $\F_2^n$-representations, respectively.
    Then, there exists a matrix $L\in \bbM{n}{n}{\rZ{4}}$ such that $D_2=\rho(L)(D_1)$ if and only if
    \begin{equation}\label{eq:M(x)*hM(y)=M(x*hhy)}
    M(x)*_{h_2} M(y)=M(x*_{h_1} y),
    \end{equation}
    where $M$ is defined by $(L \bmod 2)$ acting as an element of $\bbM{n}{n}{\F_2}$.
\end{theorem}
\begin{proof}
    By abuse of notation, we also use $\Psi$ to denote a mapping from $\bbM{n}{n}{\F_2}$ to $\bbM{n}{n}{\rZ{4}}$, which acts on every entry of the matrix as the embedding $\psi :\F_2\rightarrow \rZ{4}$ with $\psi(0)=0$ and $\psi(1)=1$.

    Let $L\in\bbM{n}{n}{\rZ{4}}$ be such that $\alpha := \rho(L)$ is an automorphism. By Lemma \ref{le:aut(z4)}, $U:= (L \mod 2)$ as an element of $\bbM{n}{n}{\F_2}$ is invertible. Clearly there is $V\in\bbM{n}{n}{\F_2}$ such that $L=U^\Psi+2V^\Psi$. Then we have
    \begin{align}
        \alpha(\lfloor a,b \rfloor)
        &= \alpha(a^\Psi+2b^\Psi) \nonumber\\
        &= (U^\Psi+2V^\Psi)(a^\Psi+2b^\Psi)^T\nonumber\\
        &= U^\Psi a^{\Psi T} +2(U^\Psi b^{\Psi T} + V^\Psi a^{\Psi T} ). \label{eq:alpha(a,b)}
    \end{align}
    Let $u_{ij}$ denote the $(i,j)$ entry of $U$. Then by \eqref{eq:basic_psi}, the $k$-th entry of $U^\Psi a^{\Psi T} $ is
    \begin{align*}
        \sum_{i=0}^{n-1} u_{ki}^\psi  a_i^\psi
        &=  u_{k0}^\psi a_0^\psi+u_{k1}^\psi a_1^\psi +\sum_{i=2}^{n-1}  u_{ki}^\psi a_i^\psi\\
        &= (u_{k0} a_0 + u_{k1}a_1 )^\psi+2(u_{k0}u_{k1}a_0a_1)^\psi+\sum_{i=2}^{n-1} u_{ki}^\psi a_i^\psi \\
        &= (u_{k0}a_0 + u_{k1}a_1 + u_{k2}a_2 )^\psi+\\
        &\phantom{=} + 2(u_{k0}u_{k1}a_0a_1+ u_{k0}u_{k2}a_0a_2+ u_{k1}u_{k2}a_1a_2)^\psi+\sum_{i=3}^{n-1} u_{ki}^\psi a_i^\psi \\
        &=(u_{k0}a_0+\dots + u_{k(n-1)}a_{n-1})^\psi+2(\sum_{i<j}u_{ki}u_{kj}a_ia_j)^\psi.
    \end{align*}
    It follows that
    $$U^\Psi a^{\Psi T}=(Ua^T)^\Psi+2(Q(U,a))^\Psi,$$
    where the $k$-th coordinate of $Q(U,a)$ is $Q_k(U,a)=\sum_{i<j}u_{ki}u_{kj}a_ia_j$. Now \eqref{eq:alpha(a,b)} becomes
    \begin{align}
      \alpha(\lfloor a,b \rfloor) =& ( U a^T)^\Psi + 2 (Q(U,a))^\Psi + \nonumber\\
       & +2\left((U b^T)^\Psi+ 2 (Q(U,b))^\Psi+(V a^T)^\Psi+ 2 (Q(V,a))^\Psi\right)\nonumber\\
      =& ( U a^T)^\Psi + 2 ((U b^T)^\Psi + (V a^T)^\Psi + (Q(U,a))^\Psi)\nonumber\\
      =& ( U a^T)^\Psi + 2 (U b^T+ V a^T+ Q(U,a))^\Psi\nonumber\\
      =& \lfloor U a^T, U b^T+ V a^T+ Q(U,a)\rfloor \label{eq:alpha(a,b)_more}
    \end{align}

    Let $M$ and $N$ be linear mappings from $\F_2^n $ to itself, which are defined by $M(x):= U x^T$ and $N(x):= V x^T$, respectively. It follows from \eqref{eq:alpha(a,b)_more} that
    \begin{equation}\label{eq:strong_iso_p=2}
      \alpha(\lfloor x,h_1(x)\rfloor)=\lfloor M(x), M (h_1(x)) + N(x) + Q(M,x)\rfloor,
    \end{equation}
    for any $x\in\F_2^n$.

    ``$\Rightarrow$'' Now we assume $D_2=\alpha(D_1)$, which means that for a given $y\in\F_2^n$ there is a unique $x\in\F_2^n$ such that
    \begin{equation*}
      \lfloor y,h_2(y)\rfloor=\alpha(\lfloor x,h_1(x)\rfloor).
    \end{equation*}
    Together with \eqref{eq:strong_iso_p=2}, it becomes
    \begin{equation*}
      \lfloor y,h_2(y)\rfloor=\lfloor M(x), M (h_1(x)) + N(x) + Q(M,x)\rfloor.
    \end{equation*}
    It follows that
    \begin{equation}\label{eq:tilde_h}
      h_2(M(x)) = M (h_1(x))+ N(x) + Q(M,x).
    \end{equation}
    Let $M_k(x\odot y)$ be the $k$-th coordinate of $M(x\odot y)$. Noticing that
    \begin{align*}
         &Q_k(M,x+ y)+ Q_k(M,x)+ Q_k(M,y)+ M_k(x\odot y)\\
        =&\sum_{i<j}a_{ki}a_{kj}((x_i+ y_i)(x_j+ y_j)+ x_ix_j+ y_iy_j)+ \sum_{i}a_{ki}x_iy_i\\
        =&\sum_{i}a_{ki}x_i\sum_{j}a_{kj}y_j\\
        =&(M(x)\odot M(y))_k,
    \end{align*}
    together with \eqref{eq:*} and \eqref{eq:tilde_h}, we have
    \begin{align*}
         &M(x)*_{h_2}M(y)\\
        =& h_2(M(x+ y))+ h_2(M (x))+ h_2(M(y))+ (M(x)\odot M(y))\\
        =& M (h_1(x+ y))+ N(x+ y) + Q(M,x+ y)+\\
        &+ M (h_1(x))+ N(x) + Q(M,x)+\\
        &+ M (h_1(y))+ N(y) + Q(M,y)+ (M(x)\odot M(y))\\
        =&M(h_1(x+ y)+ h_1(x)+ h_1(y)) + \\
        &+ \left(Q(M,x+ y)+ Q(M,x)+ Q(M,y)+ (M(x)\odot M(y))\right)\\
        =& M(h_1(x+ y)+ h_1(x)+ h_1(y)) + M(x\odot y)\\
        =& M(x*_{h_1} y).
    \end{align*}

    ``$\Leftarrow$'' Assume that there is linear mapping $M:\F_2^n\rightarrow\F_2^n$ such that \eqref{eq:M(x)*hM(y)=M(x*hhy)} holds. Let $U\in\bbM{n}{n}{\F_2}$ be such that $Ux^T=M(x)$ for any $x\in\F_2^n$. Let $D'_2:=\alpha'(D_1)$ with $\alpha' := \rho(U)$, and let $h'_2$ be the $\F_2^n$-representation of $D'_2$. Similarly to the proof of ``$\Rightarrow$'' part, we have
    \begin{equation}\label{eq:h'_2(M(x))}
        h'_2(M(x)) = M (h_1(x))+ Q(M,x),
    \end{equation}
    and
    $$M(x)*_{h'_2} M(y)=M(x*_{h_1} y).$$
    Furthermore as \eqref{eq:M(x)*hM(y)=M(x*hhy)} also holds, we get $x*_{h'_2} y= x*_{h_2} y$, i.e.\
    $$h'_2(x+ y)+ h'_2(x)+ h'_2(y)=h_2(x+ y)+ h_2(x)+ h_2(y),$$
    which implies that
    $$(h_2+ h'_2)(x+ y)+ (h_2+ h'_2)(x)+ (h_2+ h'_2)(y)=0,$$
    for all $x,y$. Hence $N'(x) := h'_2(x)+ h_2(x)$ is an additive function on $\F_2^n$.  Letting $y:= M(x)$, by \eqref{eq:h'_2(M(x))} we have
    \begin{align*}
        \lfloor y, h_2(y) \rfloor &=\lfloor y, h'_2(y)+  N'(y) \rfloor \\
        &=\lfloor M(x), M(h_1(x))+ Q(M,x)+  N'(M (x))\rfloor.
    \end{align*}
    Let $N := N'M$, and let $V\in\bbM{n}{n}{\F_2}$ be such that $Vx^T=N(x)$ for any $x\in \F_2^n$. By taking $L:= U^\Psi+2V^\Psi$ and $\alpha :=\rho(L)$, we see that \eqref{eq:strong_iso_p=2} holds, which means
    $$\lfloor y,h_2(y)\rfloor=\alpha(\lfloor x,h_1(x)\rfloor)$$
    for any $x\in\F_2^n$ and $y=M(x)$.
    Therefore we have  $D_2=\alpha(D_1)$.
\end{proof}

It is straightforward to show that:
\begin{proposition}\label{pr_equiRDS_2}
    Let $D$ be a $(2^n,2^n,2^n,1)$-RDS in $\Z_4^n$ relative to $\Z_2^n$, let $\lfloor a,b \rfloor\in\Z_4^n$ and let $h, \tilde{h}:\F_2^n\rightarrow \F_2^n$ be the $\F_2^n$-representations of $D$ and $D+\lfloor a,b \rfloor$, respectively. Then we have
    $$x*_{\tilde {h}} y=h(x+y+a)+h(x+a)+h(y+a)+b+x\odot y,\hbox{~~for }x,y\in \F_2^n.$$
\end{proposition}

\section{Coordinatization of $\Pi$}\label{se_coordination}

Let $D$ be a $(2^n,2^n,2^n,1)$-RDS in $\cy{4}^n$ relative to $N=\cy{2}^n$, and let $h$ be an $\F_2^n$-representation of $D$ with $h(0)=0$ (if $h(0)\neq 0$, then take $D-\lfloor 0,h(0) \rfloor$ instead of $D$). Let $\Pi$ be the plane defined by $D$ in Ganley's manner in Section \ref{se_introduction}. By using the method of Hughes and Piper in Chapter V of \cite{hughes_projective_1973}, we label the points of $\Pi$ by the elements of $\F_2^n\times \F_2^n$, $\F_2^n$ and by the symbol $\infty$.
\begin{enumerate}[(i)]
  \item Take three lines $L_x := D$, $L_y := N$ and $L_\infty$ to form a triangle, and label three points: $(0,0):=L_x\cap L_y$, $(\infty):= L_y\cap L_\infty$ and $(0) := L_x \cap L_\infty$;
  \item For $\lfloor x,h(x) \rfloor\in D$, label the corresponding point on $L_x$ by $(x,0)$;
  \item Assign $(1)$ to the intersection point $J$ of $L_\infty$ and the affine parallel class $\{D+\lfloor 1,k \rfloor: k\in \F_2^n\}$ (here $1\in \F_2^n$ is the vector $(0,0,\dots,1)$ for short);
  \item Let $X=(x,0)$ be a point on $L_x$ and label $JX\cap L_y$ by $(0,x)$. In fact, $JX$ is the set $D+\lfloor 1,k \rfloor$ for some $k\in \F_2^n$, which is determined by the equation $(D+\lfloor 1,k \rfloor)\cap D=(x, h(x))$. By some calculations we get
      $$(D+\lfloor 1,k \rfloor)\cap N=\lfloor 0, (1*x) \rfloor,$$
      where $*:= *_h$ is defined as in (\ref{eq:*});

  \item For each line through $(1,0)$, which intersects $L_y$ at $(0,m)$, assign $(m)$ to its intersection with $L_\infty$, which is on $D+\lfloor v,k \rfloor$ with $v*1=m$;
  \item For each point $E$ not on $L_x$, $L_y$ or $L_\infty$, if $XE\cap L_y$ is $(0,y)$ and $YE\cap L_x$ is $(x,0)$, then $E$ is given the coordinate $(x,y)$.
\end{enumerate}

Let $\star$ be the multiplication of the planar ternary ring of $\Pi$, which is coordinated as above. Then it is routine to check that
\begin{equation}\label{eq:star}
  m\star x= \tau(m *x),
\end{equation}
where $\tau: 1*x\mapsto x$, and the identity element of $(\F_2^n,\star,+)$ is $1$.
\begin{remark}
    The above method of labeling points of $\Pi$ are slightly different from the original method by Hughes and Piper: we first label the points on $L_x$ instead of $L_y$. However, these two different methods lead to the same labeling of points. Therefore, by the corollary of Theorem 6.3 in \cite{hughes_projective_1973}, $L_\infty$ is a translation line if and only if $\star$ satisfies the left distributive law, namely that $(\F_2^n,+,\star)$ is a quasifield.
\end{remark}

Next we consider the conditions under which $\Pi$ is a semifield plane.
\begin{theorem}\label{th_semifield}
    Let $D$ be a $(2^n,2^n,2^n,1)$-RDS  in $\Z_4^n$ relative to $\Z_2^n$.
    Let $h:\F_2^n\rightarrow \F_2^n$ be the normalized $\F_2^n$-representation of $D$,and let $*:=*_h$ and $\star$ be defined by (\ref{eq:*}) and (\ref{eq:star}), respectively. Let $f_B(x)$ be the normalized $\F_{2^n}$-representation of $D$ with respect to a basis $B$. Then the following statements are equivalent:
    \begin{enumerate}
        \item $(\F_2^n,\star,+)$ is a commutative semifield;
        \item $(\F_2^n,*,+)$ is a commutative presemifield;
        \item $h(x+y+z)+h(x+y)+h(x+z)+h(y+z)+h(x)+h(y)+h(z)=0$, for all $x,y,z\in \F_2^n$;
        \item Every component function of $h(x)$ is of degree at most $2$;
        \item $f_B(x)$ is a Dembowski-Ostrom polynomial;
        \item $\Pi$ is a commutative semifield plane.
    \end{enumerate}
\end{theorem}
\begin{proof}
    We only need to prove the distributivity of $*$ and $\star$  for one side by their commutativity.

    Assume that $*$ defines the  multiplication of a presemifield. Then $x\mapsto 1*x$ is an additive mapping by the distributivity of $*$, so is its inverse $\tau: 1*x \mapsto x$. Therefore we have
    \begin{align*}
         & (x+y)\star z-x\star z-y\star z\\
        =& \tau((x+y)*z)-\tau(x*z)-\tau(y*z)\\
        =& \tau(x*z+y*z-x*z-y*z)\\
        =& 0,
    \end{align*}
    for all $x,y,z\in \F_2^n$, from which we deduce the distributivity of $\star$, i.e. $(2)\Rightarrow(1)$.

    Next, we assume that $\star$ defines the multiplication of a semifield. Let $g(x):=x\star x$. Since $x*x=x$ we get
    $$g(x)=\tau(x*x)=\tau(x),$$
    and
    \begin{align*}
        g(x+y)
        =&(x+y)\star (x+y)\\
        =&x\star x+x\star y+ y\star x+y\star y\\
        =&x\star x+y\star y\\
        =&g(x)+g(y)
    \end{align*}
    by the distributivity and commutativity of $\star$. Hence $\tau$ is also an additive mapping, so is its inverse. Therefore, $(x+y)*z=\tau^{-1}((x+y)\star z)=\tau^{-1}(x\star z+y\star z)=x*z+y*z$, i.e. $(1)\Rightarrow(2)$

    The equivalence between $(2)$ and $(3)$ follows directly from the expansion of $(x+y)*z=x*z+y*z$;

    $(4)\Rightarrow (2)$: If $(4)$ holds, then we see that for fixed $y\neq 0$ every component function of $x\mapsto x*y$ is an additive mapping from $\F_2^n$ to $\F_2$. It implies the distributive property of the multiplication $*$.

    $(2)\Rightarrow (4)$: Consider any component function of $x*y$, which defines a bilinear form $B(x,y):\F_2^n\times\F_2^n \rightarrow \F_2$. By the relationship between quadratic forms and bilinear forms, wee see that the corresponding component functions of $h$ must be of degree $2$ at most;

    $(4)\Leftrightarrow(5)$: As $f_B(x)$ and $h(x)$ are both normalized, $f_B(x)$ is a Dembowski-Ostrom polynomial if and only if all components of $h(x)$ are of degree $\le 2$ (here $0$ is also considered as a Dembowski-Ostrom polynomial);

    $(1)\Rightarrow(6)$ is directly from the definition;

    $(6)\Rightarrow(1)$: Assume that $\Pi$ is a commutative semifield plane, i.e. by coordinatizing $\Pi$ in an appropriate way, a commutative semifield $(\F_2^n,\diamond ,+)$ can be obtained. Now we label the points and lines of $\Pi$ in another way which is different from the coordinatization at the beginning of this section. We use $(x,y)_\diamond$ to denote the affine points of $\Pi$, where $x,y\in \F_2^n$, and we use point sets
    $$[m,k]_\diamond:=\{(x,y)_\diamond: m\diamond x+y=k\}\hbox{ with }m,k\in\F_2^n$$
    and
    $$[k]_\diamond:=\{(k,y)_\diamond: y\in\F_2^n\}\hbox{ with }k\in\F_2^n$$
    to denote the affine lines. Every parallel class of affine lines corresponds to a point, and all such points form the line $\tilde{L}_\infty$ of $\Pi$.
    There are $4^n$ bijections
    $$\alpha_{ab}: (x,y)_\diamond\mapsto(x+a,y+a\diamond x+b)_\diamond\hbox{ with }a,b\in\F_2^n,$$
    which are collineations of $\Pi$ and together form a shift group $\tilde{G}$.
    Furthermore, by \cite[Theorem 9.4]{knarr_polarities_2009}, every shift group of $\Pi$ is of the form
    $$\tilde{G}_s := \{(x,y)_\diamond\mapsto (x+a, y+(a\diamond s)\diamond x+b)_\diamond| a,b\in \F_2^n\},$$
    where $s$ belongs to the middle nucleus of $(\F_2^n,\diamond ,+)$. It follows that there exists $s_0$, such that $\tilde{G}_{s_0}$ and $G$ act on $\Pi$ in the same way. It is routine to check that $\tilde{L}_\infty$ is the unique line fixed by $\tilde{G}_{s_0}$. Since $L_\infty$ is also fixed by $G$, we see that $\tilde{L}_\infty$ and $L_\infty$ are the same line.  On the other hand, by the distributive law of $(\F_2^n,\diamond ,+)$, we have the collineations
    $$\beta_{ab}: (x,y)_\diamond\mapsto(x+a,y+b)_\diamond\hbox{ with }a,b\in\F_2^n,$$
    which act regularly on the affine points of $\Pi$ and fix the line $\tilde{L}_\infty$ pointwise. It implies that $\tilde{L}_\infty=L_\infty$ is a translation line. Therefore, by the corollary of Theorem 6.3 in \cite{hughes_projective_1973}, $(\F_2^n,\star,+)$ satisfies the left distributive law. Together with the commutativity of $\star$, we see that $(\F_2^n,\star,+)$ is a commutative semifield.
\end{proof}

\begin{corollary}\label{co_D1equivD2withoutConstant}
    Let $D_1$ and $D_2$ be two $(2^n,2^n,2^n,1)$-RDSs in $\Z_4^n$ relative to $\Z_2^n$, which define commutative semifields and $0$ is in both $D_1$ and $D_2$. If there exists $\alpha\in\Aut(\Z_4^n)$ and $g\in\Z_4^n$ such that $\alpha(D_1)=D_2+g$, then there is also some $\beta\in\Aut(\Z_4^n)$ such that $\beta(D_1)=D_2$.
\end{corollary}
\begin{proof}
    Let $h_i$ be the $\F_2^n$-representation of $D_i$, $i=1,2$. Let $\lfloor a, b\rfloor=g$. Since $0=\lfloor 0,0\rfloor\in D_1, D_2$, we have $b=h_2(a)$. Let $\tilde{h}_2$ be the $\F_2^n$-representation of $D_2+g$. Then by Proposition \ref{pr_equiRDS_2} we have
    \begin{align*}
    x*_{\tilde {h}_2} y &=h_2(x+y+a)+h_2(x+a)+h_2(y+a)+h_2(a)+x\odot y\\
                        &=h_2(x+y)+h_2(x)+h_2(y)+x\odot y\\
                        &=x*_{h_2} y,
    \end{align*}
    in which the second equality comes from Theorem \ref{th_semifield}$(3)$. Let $\alpha :=\rho(L)$ and let $M$ be the linear mapping defined by $(L \mod 2)$ (see Theorem \ref{th_equiRDS}). We have
    $$M(x*_{h_1} y)=M(x)*_{\tilde{h}_2} M(y)=M(x)*_{h_2} M(y).$$
    Thus by Theorem \ref{th_equiRDS}, there is some $\beta\in\Aut(\Z_4^n)$ such that $\beta(D_1)=D_2$.
\end{proof}

By Theorem \ref{th_semifield}, we can obtain a non-translation plane by a non-Dembowski-Ostrom polynomial $f(x)$ which acts as a planar function on $\F_{2^n}$. Hence we propose the following open problem:
\begin{problem}\label{qu_nonquadratic}
    Find a non-Dembowski-Ostrom planar polynomial $f(x)\in\F_{2^n}[x]$, or prove the nonexistence of it.
\end{problem}

Let $(\eS_1, *, +)$ and $(\eS_2, \star, +)$ be two (pre)semifields with the same cardinality $p^n$. They are called \emph{isotopic} if there are linear bijective mappings $M,N,L:\F_p^n \rightarrow \F_p^n$ such that
$$M(x)*N(y)=L(x\star y).$$
Furthermore, if $M=N$, then $\eS_1$ is \emph{strongly isotopic} to $\eS_2$.
The isotopism is the most important equivalence relation between (pre)semifield, since Albert \cite{Albert1960} showed that two (pre)semifields coordinate isomorphic planes if and only in they are isotopic. By isotopism we can also get a semifield $\mathbb{S}$ from a presemifield $\mathbb{P}$. Let $*$ be the multiplication of a presemifield. Then for every $e\in\mathbb{P}$, we obtain a semifield multiplication $\star$ defined by:
$$(x*e)\star(y*e)=x*y,$$
with the identity $e*e$. If $*$ is commutative, then we can also use (\ref{eq:star}) to define a semifield multiplication.

For $p=2$, there is a result on the commutative (pre)semifields obtained by Coulter and Henderson.

\begin{lemma}[\cite{Coulter2008}, Corollary 2.7]\label{le:CoulterHenderson}
    Two commutative presemifields of even order are isotopic if and only if they are strongly isotopic.
\end{lemma}

By Theorem \ref{th_semifield}, we have seen how to get a semifield plane from an RDS. On the other hand, given a semifield plane, an RDS containing $0$ can also be produced, see \cite{ghinelli_finite_2003} by Ghinelli and Jungnickel.

Now assume that $D_1$ and $D_2$ are two $(2^n,2^n,2^n,1)$-RDS in $\Z_4^n$ obtained from two commutative semifields $\eS_1$ and $\eS_2$ both of order $2^n$. We use $h_1$ and $h_2$ to denote their $\F_2^n$-representations respectively. Then we have $h_1(0)=h_2(0)=0$ since $0\in D_1$ and $D_2$. If $D_1$ and $D_2$ are equivalent, then by Corollary \ref{co_D1equivD2withoutConstant} there exists $\alpha\in\Aut(\Z_4^n)$ such that $\alpha(D_1)=D_2$. By Theorem \ref{th_equiRDS}, $\eS_1$ and $\eS_2$ are isotopic.

On the contrary, if $\eS_1$ and $\eS_2$ are isotopic, then by Lemma \ref{le:CoulterHenderson} they are strongly isotopic, i.e.\ there exist $M,L :\F_2^n \rightarrow \F_2^n$ such that
$$M(x)*_{h_1}M(y)=L(x*_{h_2}y),$$
which is
\begin{align*}
    &h_1\circ M(x+y)+h_1\circ M(x)+h_1\circ M(y)+M(x)\odot M(y)\\
    =&L(h_2(x+y)+h_2(x)+h_2(y)+x\odot y) .
\end{align*}
Notice that for each quadratic function $h:\F_2^n \rightarrow \F_2^n$ and $i=0,1,\cdots,n-1$, the term $x_iy_i$ can never appear in the components of $h(x+y)+h(x)+h(y)$, hence
$$M(x)\odot M(y)=L(x\odot y).$$
Let $x=y=e_i$, which denotes the vector with a $1$ in the $i$-th coordinate and $0$'s elsewhere. We have
$$M(e_i)=M(e_i)\odot M(e_i)=L(e_i\odot e_i)=L(e_i)$$
for each $i=0,1,\cdots, n-1$, which means that $M=L$. Thus
$\left(M(x)*_{h_1}M(y)\right)=M(x*_{h_2} y).$
By Theorem \ref{th_equiRDS}, we see that $D_1$ and $D_2$ are equivalent. We summarize the above results in the following proposition.
\begin{proposition}
    Let $\eS_1$ and $\eS_2$ be two commutative semifields of order $2^n$. Let $D_1$ and $D_2$ be the RDSs derived from them. Then $\eS_1$ is isotopic to $\eS_2$ if and only if $D_1$ is equivalent to $D_2$.
\end{proposition}

\begin{remark}
    For $p$ odd, the equivalence between RDSs from commutative semifields is equivalent to the strong isotopism between the commutative semifields. As shown by Coulter and Henderson \cite{Coulter2008}, by isotopism one semifield can produce at most two semifields which are not strongly isotopic. This upper bound can be achieved. Examples can be found in \cite{pieper-seier_remarks_1999,zhou_new_2013}.
\end{remark}

\section{Nonexistence of Boolean Planar Function }\label{se_nonexistence}

By Example \ref{ex_KWsemifield}, we see that there are many planar functions on $\F_{2^m}$, where $m$ has at least one odd divisor larger than 1. Actually in \cite{kantor2003}, Kantor proved that the number of pairwise non-isotopic semifields of order $2^m$ is not bounded by a polynomial in $N=2^m$.

In order to find  a non-Dembowski-Ostrom planar function, one strategy is to do some ``small modifications'' to known planar functions, which preserve the planar property. One method is called ``switching construction'', in which one tries to change just one coordinate function of $f$. This method was first applied on APN function on $\F_{2^n}$ \cite{dillon_APNpermutation_2010,EdelPott2009}, providing a counter example to the ``APN permutation conjecture''. This idea can also be applied to planar functions in characteristic 3, see \cite{pott_switching_2010}.

In this section, we consider the Boolean planar functions, or more generally, planar functions $f$ with $\mathrm{Im}(f)=\{0,\xi\}$. These can be considered as a ``small modification'' of the planar function $g=0$. Actually, we can prove the following theorem, which implies that such planar functions $f$ are always ``trivial''.
\begin{theorem}\label{th_linear}
    Let $n$ be a positive integer and $f$ be a mapping on $\F_{2^n}$ where $f(0)=0$ and $\mathrm{Im}(f)=\{0,\xi\}$ with $\xi\neq 0$. Then $f$ is a planar mapping if and only if $f$ is additive, i.e.\ $f(x+y)=f(x)+f(y)$ for any $x,y\in\F_{2^n}$.
\end{theorem}

To prove Theorem \ref{th_linear}, we need several lemmas.
\begin{lemma}\label{le:planar_iff}
    Let $f$ be a mapping on $\F_{2^n}$, let $\xi\in\F_{2^n}^*$ and suppose that $\mathrm{Im}(f)=\{0,\xi\}$. Then $f$ is a planar mapping, if and only if, for all $a\neq 0$ and $x$
    \begin{equation}\label{eq:fx is planar}
      f(x+a)+f(x)+f(x+a+\frac{\xi}{a})+f(x+\frac{\xi}{a})=0.
    \end{equation}
\end{lemma}
\begin{proof}
    If $f$ is planar, then for all $a\neq 0$,
    $$f(x+a)+f(x)+ax\neq f(x+a+\frac{\xi}{a})+f(x+\frac{\xi}{a})+ax+\xi,$$
    which is equivalent to
    $$f(x+a)+f(x)+f(x+a+\frac{\xi}{a})+f(x+\frac{\xi}{a})\neq \xi.$$
    As $\mathrm{Im}(f)\in\{0,\xi\}$, we obtain \eqref{eq:fx is planar}.

    Next we suppose that $f$ is not a planar function. By definition,there exist $a,x,y\in\F_{2^n}$ satisfying $x\neq y$ and $a\neq 0$ such that
    $$f(x+a)+f(x)+xa=f(y+a)+f(y)+ya.$$
    Since $\mathrm{Im}(f)=\{0,\xi\}$ and $xa\neq ya$, we have $xa+\xi=ya$, which means that $y=x+\xi/a$ and
    \begin{equation*}
        f(x+a)+f(x)+f(x+a+\frac{\xi}{a})+f(x+\frac{\xi}{a})=\xi. \qedhere
    \end{equation*}
\end{proof}

Given $f:\F_{2^n} \rightarrow \F_{2^n}$, we define
\begin{equation}
    \mathcal{A}_f:= \{(a,b): f(x+a)+f(x)+f(x+b+a)+f(x+b)=0\}.
\end{equation}
The set $\{0\}\times \F_{2^n},\F_{2^n}\times \{0\}$ and $\{(a,a):a\in\F_{2^n}\}$ are all contained in $\mathcal{A}_f$. We can also prove the following relations between the elements of $\mathcal{A}_f$.
\begin{lemma}\label{le:rho_wedge_vee}
    Let $f:\F_{2^n} \rightarrow \F_{2^n}$. There are binary operations $\wedge,\vee:\F_{2^n}\times\F_{2^n} \rightarrow \F_{2^n}$ such that
    \begin{enumerate}
        \item If $(a,b),(a+b,c)\in\mathcal{A}_f$, then $(a,b)\wedge(a+b,c)=(a+b,b+c)\in\mathcal{A}_f$;
        \item If $(a,b),(a,c)\in\mathcal{A}_f$, then $(a,b)\vee(a,c)=(a,b+c)\in\mathcal{A}_f$.
    \end{enumerate}
\end{lemma}
\begin{proof}
    We just prove the first case. Since $(a,b),(a+b,c)\in\mathcal{A}_f$, we have
    \begin{eqnarray*}
      f(x+a)+f(x)+f(x+b+a)+f(x+b) &=& 0 \\
      f(x+a+b)+f(x)+f(x+c+a+b)+f(x+c) &=& 0 .
    \end{eqnarray*}
    Summing these equations, we get
    $$f(x+c+a+b)+f(x+c)+f((x+c)+(b+c)+(a+b))+f((x+c)+(b+c))=0,$$
    which means that $(a+b,b+c)\in\mathcal{A}_f$.
\end{proof}

If $f$ is a planar function satisfying $\mathrm{Im}(f)=\{0,\xi\}$, then we have $\{(a,\xi/a):a\in\F_{2^n}^*\}\subseteq\mathcal{A}_f$ by Lemma \ref{le:planar_iff}. Thus, for $(a_0,\xi/a_0)\in\mathcal{A}_f$ with $a_0\neq 0$, we also have $(a_0+\frac{\xi}{a_0},\frac{\xi}{a_0+\frac{\xi}{a_0}})\in\mathcal{A}_f$. Hence we can define $b_0:= {\xi}/{a_0}$ and
$$(a_{i+1},b_{i+1}):=(a_i,b_i)\wedge(a_i+b_i,\frac{\xi}{a_i+b_i})=(a_i+b_i,b_i+\frac{\xi}{a_{i+1}}),$$
and all these $(a_i,b_i)$ are contained in $\mathcal{A}_f$ by Lemma \ref{le:rho_wedge_vee} (1).
As it is possible that $a_i+b_i=0$, by abuse of notation we define $\frac{x}{0}:=0$. It is easy to see that $$(a_{i},b_{i}):=(a_{i+1}+\frac{\xi}{a_{i+1}}+b_{i+1}, b_{i+1}+\frac{\xi}{a_{i+1}}),$$
which means that the sequence $((a_i,b_i): i=0,1,\cdots)$ is cyclic. Furthermore, we have
$$a_{i+1}=a_i+b_i=a_{i-1}+\frac{\xi}{a_i},$$
and $$b_i=a_{i}+a_{i+1}=a_{i-1}+a_i+\frac{\xi}{a_i},$$
from which we can deduce that $a_{i-1},a_{i}$ determine $b_i$, hence the period of $((a_i,b_i): i=0,1,\cdots)$ is the same as the period of  $(a_i: i=0,1,\cdots)$.
\begin{lemma}\label{le:seq_an}
    Let $a\in\F_{2^n}$ and $a\neq 0,\sqrt{\xi}$. Let $a_0:=a$, $a_1 := a+\xi/a$ and define sequence $S_a=(a_i: i=0,1,\cdots)$ by
    \begin{equation}\label{eq:an}
      a_{i+1} :=a_{i-1}+\frac{\xi}{a_i}.
    \end{equation}
    Then
    \begin{equation}\label{eq:final_a_2i_a_2i+1}
        a_{2i}=a\left(\frac{a^2}{a^2+\xi}\right)^i,\quad a_{2i+1}=a\left(\frac{a^2+\xi}{a^2}\right)^{i+1},
    \end{equation}
    and the period of $S_a$ is $2\cdot\mathrm{ord}(1+\xi/a^2)$.
\end{lemma}
\begin{proof}
    We first prove that there is no $i$ such that $a_i=0$.

    Now assume that $a_i=0$. Then by (\ref{eq:an}) we see that $a_{i-1}=a_{i+1}$. Define $\lambda := a_{i-1}$. Then we have $$(a_{i-1},a_i,a_{i+1}, a_{i+2},a_{i+3},a_{i+4},a_{i+5},a_{i+6})=(\lambda,0,\lambda,\frac{\xi}{\lambda},0,\frac{\xi}{\lambda},\lambda,0).$$
    If $\lambda=0$ or $\sqrt{\xi}$, then every entry of $S_a$ is $0$ or $\sqrt{\xi}$, which contradicts our assumption that $a_0=a\neq 0,\sqrt{\xi}$.
    Now we need to find the values of $a_0$ and $a_1$. It is easy to check that all the possible values lead to $a=0$ or $\sqrt{\xi}$. Therefore, we proved that $a_i\neq 0$ for all $i$.

    Since every $a_i$ is nonzero, by (\ref{eq:an}) we have
    $$a_{i+1}a_i=a_ia_{i-1}+\xi.$$
    Since $a_0a_1=a^2+\xi$, we then find that
    \begin{equation*}
        \frac{a_0}{a_{2i}}=\frac{a_0a_1}{a_1a_2}\frac{a_2a_3}{a_3a_4}\cdots\frac{a_{2i-2}a_{2i-1}}{a_{2i-1}a_{2i}}
        =\left(\frac{a^2+\xi}{a^2}\right)^i
    \end{equation*}
    and
    \begin{equation*}
        \frac{a_1}{a_{2i+1}}=\left(\frac{a^2}{a^2+\xi}\right)^i.
    \end{equation*}
    Therefore we get (\ref{eq:final_a_2i_a_2i+1}) and the period of $S_a$ is $2\cdot\mathrm{ord}(1+\xi/a^2)$.
\end{proof}

\begin{lemma}\label{le:linear=Af}
    Let $f$ be a mapping from $\F_{2^n}$ to itself satisfying $f(0)=0$. Then $f$ is an additive mapping if and only if $\mathcal{A}_f=\F_{2^n}\times\F_{2^n}$.
\end{lemma}
\begin{proof}
    If $f$ is additive, then for each $(a,b)\in\F_{2^n}\times\F_{2^n}$ we have
    $$
      f(x+a)+f(x)+f(x+b+a)+f(x+b)=f(a)+f(a)=0.
    $$
    Hence $\mathcal{A}_f=\F_{2^n}\times\F_{2^n}$.

    If $\mathcal{A}_f=\F_{2^n}\times\F_{2^n}$, then for each given $a$,
    $$f(x+a)+f(x)+f(y+a)+f(y)=0,$$
    which means the mapping $x\mapsto f(x+a)+f(x)$ is constant. Since $f(0)=0$, we have $f(x+a)=f(x)+f(a)$, i.e.\ $f$ is additive.
\end{proof}

Now, we can prove a weak version of Theorem \ref{th_linear}.
\begin{reptheorem}{th_linear}
    Let $f$ be a mapping from $\F_{2^n}$ to itself with $f(0)=0$. Let $\xi$ be a nonzero element of $\F_{2^n}$. Write $$\mathcal{P}_n:=\{1/(1+\alpha): \alpha \text{ is a primitive element of }\F_{2^n}\}\cup\{1\}.$$
    If $\mathcal{P}_n$ spans $\F_{2^n}$ over $\F_2$, then the followings are equivalent:
    \begin{enumerate}
        \item $\{(a,\xi/a):a\in\F_{2^n}\}\subseteq \mathcal{A}_f$;
        \item $f$ is an additive mapping.
    \end{enumerate}
\end{reptheorem}
\begin{proof}
    By Lemma \ref{le:linear=Af}, we only need to show that, if $\{(a,\xi/a):a\in\F_{2^n}^*\}\subseteq \mathcal{A}_f$, then $\mathcal{A}_f=\F_{2^n}\times\F_{2^n}$.

    If $\alpha=1+\xi/a^2$ is a primitive element, then by Lemma \ref{le:seq_an} every $c\in\F_{2^n}^*$ appears exactly twice in $S_a$, and one of the corresponding indices of $a_i=c$ is odd, the other is even.  Fix $c$ and assume that $\alpha$ is primitive and $a_k=c$. Then
    \begin{equation*}
    (a_k,a_{k+1})
    =\left\{
    \begin{array}{ll}
      \left(c,\frac{a^2}{c}\right), & \hbox{$k$ is odd;} \\
      \left(c,\frac{a^2+\xi}{c}\right), & \hbox{$k$ is even.}
    \end{array}
    \right.
    \end{equation*}
    Since $(a_k,a_k+a_{k+1})=(a_k,b_k)\in\mathcal{A}_f$, it is easy to get that $(a_k,a_{k+1})\in\mathcal{A}_f$.
    Notice that $a^2=\xi/(1+\alpha)$. We have
    $$\left\{\left(c,\frac{\xi}{c}\frac{1}{1+\alpha}\right): \alpha \text{ is a primitive element of }\F_{2^n}\right\}\subseteq\mathcal{A}_f.$$
    Since $(c,\xi/c)$ is also in $\mathcal{A}_f$, by Lemma \ref{le:rho_wedge_vee} (2) we see that $(c,d\xi/c)\in\mathcal{A}_f$, where $d$ is a linear combination of elements of $\mathcal{P}_n$. Therefore if $\mathcal{P}_n$ spans $\F_{2^n}$, we have $\mathcal{A}_f\supseteq\F_{2^n}^*\times\F_{2^n}^*$, i.e.\;$f$ is additive.
\end{proof}

Let $\Tr_n:\F_{2^n}\rightarrow \F_2$ be the trace mapping. Notice that $\mathcal{P}_n$ spans $\F_{2^n}$ over $\F_2$ if and only if the points of $\mathcal{P}_n$ are not contained in any hyperplane of $\F_2^n$. It holds if and only if for every $\beta\in\F_{2^n}^*$, there exists some $a\in\mathcal{P}_n$ such that $\Tr_n(\beta a)=1$. This actually holds for $n\ge 18$ by setting $q=2$, $r=m=l=1$, $f_1(x)=\frac{\beta}{x+1}$ and $t_1=1$ in the following theorem.
\begin{theorem}[Stephen D. Cohen \cite{cohen_finite_2005}]\label{th_cohen}
    Let $f_1(x),\dots,f_r(x)\in\F_{q^n}(x)$ form a strongly linearly independent set over $\F_q$  with $\deg f_i\le m$, $i=1,\dots,r$ and let $t_1,\dots,t_r\in\F_q$ be given. Also let $l$ be any divisor of $q^n-1$. Suppose that
    $$n>4(r+\log_q(9.8l^{3/4}rm)).$$
    Then there exists an element $\gamma\in\F_{q^m}$ of order $(q^n-1)/l$ such that
    $$\Tr_n(f_{\gamma_i}(\gamma))=t_i,\qquad i=1,\dots,r.$$
\end{theorem}

To be \emph{strongly linearly independent} over $\F_q$ means that only the all-zero $\F_q$-linear combination of $f_1,\dots,f_r$ can be written in the form $h(x)^p-h(x)+\theta$ for some $h(x)\in\F_{q^n}(x)$ and $\theta\in\F_{q^n}$, where $p$ is the characteristic of $\F_q$. When $q=2$, $r=1$ and $f_1(x)=\frac{\beta}{x+1}$, it is readily verified that $\{f_1(x)\}$ is strongly linearly independent.

Using a MAGMA\cite{Magma} program, we showed that $\mathcal{P}_n$ also spans $\F_{2^n}$ for $n<18$. Hence, we can remove the condition on $\mathcal{P}_n$ in Theorem \ref{th_linear}* and we finish the proof of Theorem \ref{th_linear}.

\section{Component functions of $\F_2^n$-representations}\label{se_component}
Through the $\F_2^n$-representation of RDSs in $\Z_4^n$ relative to $\Z_2^n$, we see that the component functions of $h$, defined as $\sum_{i \in \Lambda} h_i$ for each nonempty subset $\Lambda\subseteq\F_2^n$, are the ingredients to build $(2^n,2^n,2^n,1)$-RDSs.

Notice that $\Delta_{h,a}(x)$ is bijective if and only if all the component functions of $\Delta_{h,a}(x)$ are balanced (the image set contains $0$ and $1\in\F_2$ equally often). We define $f$ to be a \emph{shifted-bent} (or \emph{bent$_4$}) \emph{function} with respect to $\Lambda\subseteq\{0,1,\cdots n-1\}$, if
\begin{equation}\label{eq:balanced_boolean_f}
  f(x+a)+f(x)+\sum_{i\in\Lambda}x_ia_i
\end{equation}
are balanced, for all $a\neq (0,\cdots,0)$, where $x=(x_0,\cdots,x_{n-1})$ and $a=(a_0,\cdots,a_{n-1})$. We call $\Lambda$ the \emph{shift index set} of $f$. If $\Lambda$ is empty, then shifted-bent and bent functions are the same. As in the relation between $p$-ary bent function and planar function with odd $p$, the investigation of shifted-bent functions is helpful to understand $(2^n,2^n,2^n,1)$-RDSs.

\begin{remark}
    In fact, bent$_4$ functions are defined in \cite{riera_generalized_2006} through a generalization of Walsh transformation. It is readily verified that this definition is equivalent to the definition in this paper.
\end{remark}

It is worth noting that linear and constant terms in $f$ only contribute to constant and $0$ respectively in $f(x+a)+f(x)$. Hence, they do not affect the shifted-bent property of $f$.

\begin{example}
    Let $\Lambda=\{0,1,\cdots,n-1\}$, then $f=0$ is shifted-bent.
\end{example}

Next, we give several constructions of shifted-bent functions. For large $n$, the degree of the corresponding polynomial of shifted-bent function can be larger than 2 (for small $n$, at least by computer program we find that for $n=4$ there is no non-quadratic shifted-bent function).   Let $x=(x_0, x_1, \cdots, x_{n-1})$, $y=(y_0, y_1, \cdots, y_{n-1})\in\F_2^n$. We define the inner product of $x$ and $y$ by $\langle x,y\rangle:=\sum_{i=0}^{n-1}x_iy_i$. We now give a construction, which is similar to the Maiorana-McFarland bent functions:

\begin{theorem}\label{th_MM}
    Let $g$ be an arbitrary Boolean function on $\F_2^n$, then
    $$f: (x,y) \mapsto\langle x, \Pi(y)\rangle+g(y)$$
    is shifted-bent with respect to any subset of the indices of $\{y_i:i=0,1,\cdots,n-1\}$ if and only if $\Pi$ is a permutation on $\F_2^n$.
\end{theorem}
\begin{proof}
    Let $\Lambda$ denote a subset of $\{y_i:i=0,1,\cdots,n-1\}$. Let $\Pi_i$ be the $i$-th coordinate function of $\Pi$. For $(a,b)\neq(0,0)$ we have
    \begin{align}\label{eq:MM1}
         &f(x+a,y+b)+f(x,y)+\sum_{i\in\Lambda}y_ib_i  \\
        =&\sum_{i=0}^{n-1} x_i(\Pi_i(y+b)+\Pi_i(y))+\sum_{i=0}^{n-1}a_i\Pi_i(y+b)+g(y+b)+g(y)+\sum_{i\in\Lambda}y_ib_i. \notag
    \end{align}
    When $b=0$, this equals
    $$\sum_{i=0}^{n-1}a_i\Pi(y),$$
    which is balanced for all $a\neq 0$ if and only if $\Pi$ is a permutation.

    When $b\neq 0$, we only have to show that, if $\Pi$ is a permutation, then (\ref{eq:MM1}) is balanced. Since $\Pi(y+b)+\Pi(y)\neq 0$ for each $y$, (\ref{eq:MM1}) defines a balanced Boolean function on $x\in\F_2^n$. Hence we proved the claim.
\end{proof}

By Theorem \ref{th_MM}, we have the following result:
\begin{corollary}\label{co_MM}
    Let $g$ be an arbitrary mapping from $\F_{2^n}$ to itself, and define $f:\F_{2^n}\times\F_{2^n}\rightarrow\F_{2^n}$ by
    $$f: (x,y) \mapsto x\Pi(y)+g(y).$$
    Then every component function of $f$
    is shifted-bent with respect to an arbitrary subset of the indices of $\{y_i:i=0,1,\cdots,n-1\}$ if and only if $\Pi$ is a permutation on $\F_{2^n}$.
\end{corollary}

It is easy to prove the following secondary construction of shifted-bent functions:
\begin{proposition}
    Let $m,n$ be positive integers, and let $\Pi_i:\F_2^{n}\rightarrow\F_2$ be a shifted-bent Boolean function with respect to some $\Lambda_i$, $i=1,2$. Then $$f(x,y):=\Pi_1(x)+\Pi_2(y)$$ is shifted-bent with respect to $\Lambda_1\cup\Lambda_2$.
\end{proposition}

By Theorem \ref{th_MM}, we see that there exist shifted-bent functions which are not quadratic. Furthermore, Corollary \ref{co_MM} shows us that it is possible to combine them together to a ``vectorial'' one. Recall Problem \ref{qu_nonquadratic} in Section 3, we want to find a mapping $h:\F_2^n \rightarrow \F_2^n$ (the $\F_2^n$-representation), which is an $n$-dimensional combination of shifted-bent functions $h_i$ with shift index set $\Lambda_i=\{i\}$ for $i=0,1,\dots,n-1$ and at least one of $h_i$ is of degree larger than 2. Therefore we pose this problem again in a more general way:
\begin{problem}
    What is the maximal $m$ for each $n$, such that Boolean functions $f_i:\F_2^n\rightarrow \F_2$, $i=0,1,\cdots, m-1$ satisfy the following two conditions:
    \begin{enumerate}
      \item For each non-empty set $\Omega\subseteq\{0,\dots,m-1\}$, $\sum_{i\in\Omega}f_i$ is a shifted-bent function with respect to $\Omega$;
      \item At least one of $f_i$ is non-quadratic?
    \end{enumerate}
    Hence Problem \ref{qu_nonquadratic} is about whether $m=n$ is possible for some $n$.
\end{problem}

\section*{Acknowledgement}
I would like to thank Norbert Knarr, Alexander Pott and Kai-Uwe Schmidt for valuable comments and suggestions. In addition I would like to thank Gohar Kyureghyan for pointing out Theorem \ref{th_cohen}, which leads to the proof of Theorem \ref{th_linear}. Part of this work was done while I was visiting the University of Padua, I am grateful to Michel Lavrauw and Corrado Zanella for their hospitality.

This work is partially supported by the National Science Foundation of China (Grant No. 61070215, 61103191).

\bibliographystyle{plain}

\end{document}